\documentclass[12pt]{amsart}
\usepackage[latin1]{inputenc}
\usepackage{amsmath} 
\usepackage{amsfonts}
\usepackage{amssymb}
\usepackage{stmaryrd}
\usepackage{latexsym} 
\usepackage{graphicx}
\usepackage{subfigure}
\usepackage{color}
\usepackage{hyperref}
\usepackage{verbatim}
\usepackage[all]{xy}
\usepackage{graphics}
\usepackage{pdfsync}

\usepackage{tikz}

\theoremstyle{plain}
\newtheorem{theorem}{Theorem}[section]
\newtheorem{lemma}[theorem]{Lemma}

\theoremstyle{remark}
\newtheorem{remark}[theorem]{Remark}

\title[Locally finite lattices]{Locally finite sublattices of free lattices}

\author{Brian T. Chan}
\address{
 Department of Mathematics \\
 University of British Columbia\\
 Vancouver BC V6T 1Z2, Canada}
\email{bchan@math.ubc.ca}

\date{\today}
\subjclass[2010]{06B05, 06B20, 06B25}
\keywords{locally finite posets, locally finite lattices, free lattices, sublattices of free lattices, countability conditions}
\thanks{The author was supported in part by the Natural Sciences and Engineering Research Council of Canada \includegraphics[scale = 0.2]{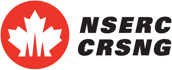} [funding reference number PGSD2 - 519022 - 2018].}

\begin{document}

\begin{abstract} The problem of determining which infinite lattices are (isomorphic to) sublattices of free lattices is in general unsolved and extremely difficult. In this note, we reduce the problem by proving that all locally finite sublattices of free lattices are countable by using a result from Baldwin, Berman, Glass and Hodges on free algebras.
\end{abstract}
\maketitle
%

\section{Introduction}

Free lattices have been the subject of much investigation within lattice theory, with Whitman introducing \emph{Whitman's condition} \cite{Whitman 1, Whitman 2} and J\'onsson introducing \emph{semidistributive lattices} to study properties of free lattices \cite{Semidistributive origins 1, Semidistributive origins 2}. An important problem within the theory of free lattices that has received a lot of attention over the years is the problem of determining, up to isomorphism, sublattices of free lattices \cite{FL}. \\

The majority of what is known about sublattices of free lattices is based on what we know about finite sublattices of free lattices, extensions are known for finitely generated sublattices of free lattices and projective lattices \cite{FL}. Finite sublattices of free lattices can be characterized by using Whitman's condition and a property involving \emph{join covers} of elements \cite{FL}. Later on, this characterization was strenghened to requiring only the semidistributive laws and Whitman's condition \cite{Nations proof}. Another aspect of free lattices that was discovered is as follows \cite{FL}. It is known that group actions on free lattices can be used to prove that any chain in a free lattice is countable \cite{Distributive sublattices} and that every free lattice is a countable union of antichains \cite{FL}. In this note, we add to what is known about sublattices of free lattices by proving that all locally finite sublattices of free lattices are countable by using a result from Baldwin, Berman, Glass and Hodges on free algebras.

\section{Preliminaries}

Let $\mathbb{N}$ denote the set of positive integers. In introducing posets and lattices, we use \cite{ILO} as a reference. A \emph{poset} $(P, \leq)$ is a set $P$ equipped with a \emph{partial order} $\leq$ on $P$ (a binary relation on $P$ that is \emph{reflexive}, \emph{antisymmetric}, and \emph{transitive}). For convenience, we usually write $P$ to denote a poset $(P, \leq)$. In particular, we write $a \in P$ and $S \subseteq P$ to denote an element of and a subset of the \emph{set} $P$ respectively in the poset $(P, \leq)$. If $(P, \leq)$ is a poset and if $a,b \in P$, then we write $a \parallel b$ to indicate that $a \leq b$ is false and that $b \leq a$ is false. We also write $a \geq b$ to mean that $b \leq a$, $a < b$ to mean that $a \leq b$ and $a \neq b$, and $a > b$ to mean that $b < a$. \\

A \emph{lattice} is a poset $L$ such that any finite subset $S$ of $L$ has a \emph{greatest lower bound} $\bigwedge S$ in $L$ and a \emph{least upper bound} $\bigvee S$ in $L$. If $S = \{a, b\}$ and $a \neq b$, then write $\bigvee S = a \vee b$ and $\bigwedge S = a \wedge b$. Moreover, for all $a \in L$, write $a \vee a = a$ and $b \wedge b = b$. For clarity, we will sometimes say \emph{in $P$} or \emph{in $L$} when describing $P$ and $L$. If $L$ is a lattice, and if $S \subseteq L$ is such that for all finite subsets $S_0$ of $S$, the least upper bound of $S_0$ in $L$ is in $S$ and the greatest lower bound of $S_0$ in $L$ is in $S$, then a \emph{sublattice of $L$} is the set $S$ equipped with the partial order defined, for all $a, b \in S$, by $a \leq b$ if $a \leq b$ in $L$. All sublattices of lattices are lattices. When describing posets and lattices, we will often just write $P$ and $L$ rather than $(P, \leq)$ and $(L, \leq)$. Hence, if $P$ is a poset, then we write $a \in P$ to denote an element of \emph{the set} $P$ and we write $S \subseteq P$ to denote a subset of \emph{the set} $P$. Moreover, if $P$ is a poset and if $S \subseteq P$, then we let $P \backslash S$ denote the set of elements in $P$ that are not in $S$. Lastly, we sometimes say \emph{in $P$} or \emph{in $L$} for clarity when describing elements or inequalities. \\

Let $P$ be a poset. If $a \in P$, then $a$ is a \emph{minimal} element of $P$ if for all $b \in P$, $b \nless a$.  An \emph{antichain} in $P$ is a subset $S \subseteq P$ such that $a \parallel b$ for all distinct $a, b \in S$. Moreover, a \emph{chain in $P$} is a subset $S \subseteq P$ such that for all distinct $a, b \in S$, $a \leq b$ or $b \leq a$. If $a, b \in P$, then $b$ \emph{covers $a$ in $P$} (or $a$ is \emph{covered by $b$ in $P$}) if $a \leq b$ in $P$ but no element $c \in P$ satisfies $a < c < b$ in $P$. We write $a \prec b$ if $a$ is covered by $b$ in $P$, and we write $a \succ b$ if $a$ covers $b$ in $P$. Furthermore, if $a, b \in P$ are such that $a \leq b$, then an \emph{interval $[a, b]$ of $P$} is the following set of elements $\{c \in P : a \leq c \leq b \text{ in } P \}$. Lastly, if $a \in P$, then write $\downarrow_P a = \{b \in P : b \leq a \}$ and write $\uparrow_P a = \{b \in P : b \geq a\}$.  \\

A poset $P$ is \emph{locally finite} if for all $a,b \in P$ such that $a \leq b$, the interval $[a,b]$ is finite. If $K$ and $L$ are lattices, then a \emph{lattice homomorphism $f : K \to L$} is a function from the set of elements of $K$ to the set of elements of $L$ such that for all $a,b \in K$, $f(a \vee b) = f(a) \vee f(b)$ and $f(a \wedge b) = f(a) \wedge f(b)$. If $K$ and $L$ are lattices, then $K$ is \emph{isomorphic to} $L$ if there exist lattice homomorphisms $f : K \to L$ and $g : L \to K$ such that $f$ and $g$ are bijections, $g \circ f$ is the identity map on $K$, and $f \circ g$ is the identity map on $L$. \\

If $S$ is a set, then a \emph{free lattice on $S$} is a lattice $FL(S)$ that satisfies the following \emph{universal property}. For all lattices $L$ and for all functions $f$ from $S$ to the set of elements of $L$, there exists a unique lattice homomorphism $g : FL(S) \to L$ such that for all $s \in S$, $g(s) = f(s)$ (\cite{FL}, p. 136). Any two free lattices on $S$ are isomorphic, so we say that $FL(S)$ is \emph{the} free lattice on $S$ (\cite{FL}, p. 136). The free lattice $FL(S)$ is also written as $FL(|S|)$. Free lattices can also be defined using equivalence classes on the set of well-formed finite words constructed with the elements of $S$, the symbols $\vee$ and $\wedge$, and parenthesis \cite{FL}. For free lattices, we will slightly abuse terminology and say that a lattice $L$ is a \emph{sublattice of a free lattice} if $L$ is \emph{isomorphic} to a sublattice of $FL(S)$ for some set $S$. \\

We will use some notions from $ZFC$ set theory \cite{Set theory text}. By \emph{maximal} subsets we mean subsets that are maximal with respect to set inclusion. Recall that $ZFC$ denotes \emph{Zermelo-Fraenkel set theory with the Axiom of Choice}, recall the notion of an \emph{ordinal number}, recall the finite ordinal numbers $0$, $1$, $2$, $\dots$, and recall the ordinal number $\omega = \{0, 1, 2, \dots\}$. In particular, we will consider finite ordinal numbers as non-negative integers and vice-versa. Moreover, recall the notion of a \emph{well-ordered set}, recall the notion of an \emph{order type}, recall that $ZF$ denotes \emph{Zermelo-Fraenkel set theory without the Axiom of Choice}, and recall that the Axiom of Choice is logically equivalent to the \emph{Hausdorff Maximal Principle} over $ZF$. A \emph{tree} is a poset $T$ in which every subset $\{y \in T : y < x \}$ is well-ordered for all $x \in T$. Let $T$ be a tree. A \emph{branch} of $T$ is a maximal chain in $T$ and the \emph{length} of a branch is the order type of that branch. Moreover, the \emph{height} of an element $x \in T$ is the order type of $\{y \in T : y < x\}$, and the \emph{height} of $T$ is the least ordinal that is greater than the height of any element of $T$. Given an ordinal $\alpha$, the \emph{$\alpha^{th}$ level} of $T$ is the set of elements of $P$ with height $\alpha$. Lastly, a \emph{root} of $T$ is a minimal element of $T$. \\

\section{Locally finite sublattices}

In this section, we prove that all locally finite sublattices of free lattices are countable using the Axiom of Choice and a result from Baldwin, Berman, Glass and Hodges on free algebras. 

\begin{theorem}(Baldwin, Berman, Glass, and Hodges \cite{Baldwin et al}) \label{free algebras} In a free lattice $FL(S)$, it is impossible for there to be an uncountable subset $Y \subseteq FL(S)$ and an element $c \in FL(S)$ such that $a \wedge b = c$ for all distinct $a,b \in Y$.
\end{theorem}

\begin{remark} Theorem \ref{free algebras} is a special case of Corollary 6 of \cite{Baldwin et al}. Moreover, in \cite{Baldwin et al}, Baldwin, Berman, Glass, and Hodges proved Theorem \ref{free algebras} by using a special case of Erd\"{o}s and Rado's \emph{$\Delta$-system Lemma} \cite{Erdos and Rado}.
\end{remark}

We first prove the following lemma.

\begin{lemma}\label{first step} Let $L$ be a locally finite sublattice of a free lattice. Moreover, assume that $L$ has exactly one minimal element. Then $L$ is countable.
\end{lemma}

\begin{proof} Assume without loss of generality that $L$ is infinite, and let $a_0$ be the unique minimal element of $L$. Because $L$ is locally finite, it follows that for all $a \in L$ such that $a \neq a_0$, there exists a positive integer $n$ and elements $a_1, a_2, \dots, a_n \in L$ such that $a_n = a$ and $a_0 \prec a_1 \prec a_2 \prec \dots \prec a_n$ in $L$. Hence, there is a sequence of finite trees $T_0'$, $T_1'$, $\dots$ such that the set of elements in $T_k'$ is a subset of $L$ for all $k \in \omega$, $L = \bigcup_{k \in \omega} T_k'$ as sets, and all of the following hold for all $k \in \omega$. The tree $T_k'$ has a unique root, $a_0$ is the root of $T_k'$, $T_k'$ is of height $k + 2$, for all $a,b \in T_k'$, $a \prec b$ in $T_k'$ implies that $a \prec b$ in $L$, and for all $a,b \in T_k'$, $a < b$ in $T_k'$ if and only if $a < b$ in $T_{k+1}'$. So by the Axiom of Choice, there exists a tree $T$ such that the set of elements in $T$ is the set of elements in $L$, $T$ has exactly one root, $a_0$ is the root of $T$, and for all $a, b \in T$, $a \prec b$ in $T$ implies that $a \prec b$ in $L$. \\

Because the length of every branch of $T$ is at most $\omega$, the height of $T$ is at most $\omega$. For all $k \in \omega$, let $T_k$ denote the $k^{th}$ level of $T$. Then, $T = \bigcup_{k \in \omega} T_k$ as sets. For all $k \in \omega$ and for all $a \in T_k$, consider the set $S_a = \{b \in T_{k+1} : a \prec b \text{ in } T \}$. As $a \prec b$ in $L$ for all $b \in S_a$, it follows that for all $b_1, b_2 \in S_a$, $b_1 \wedge b_2 = a$ in $L$. Hence, by Theorem \ref{free algebras}, $S_a$ is countable. Lastly, note that $T_0$ is countable because $T_0 = \{a_0\}$. So as the countable union of countable sets is always countable, it follows that all levels of $T$ are countable. Hence, as $T = \bigcup_{k \in \omega} T_k$ as sets, $T$ is countably infinite. Therefore, as the set of elements in $T$ equals to the set of elements in $L$, $L$ is countably infinite.
\end{proof}

We now prove the main theorem of this note.

\begin{theorem} If $L$ is a locally finite sublattice of a free lattice, then $L$ is countable.
\end{theorem}

\begin{proof} Let $L$ be a locally finite sublattice of a free lattice. Define $P$ to be the poset of antichains of $L$ that are ordered by set inclusion. By the Hausdorff Maximal Principle applied to $P$, there exists a maximal antichain $S$ in $L$. So define $L_0 = S$. If $k \in \mathbb{N}$, then let $L_k$ be the set of elements $a \in L$ with the following property. The smallest positive integer $n$ such that $a_0 \prec a_1 \prec \dots \prec a_{n-1} \prec a$ in $L$ for some $a_0 \in L_0$ and for some $a_1, a_2, \dots, a_{n-1} \in L$ is $n = k$. Dually, if $k \in \mathbb{N}$, then let $L_{-k}$ be the set of elements $a \in L$ with the following property. The smallest positive integer $n$ such that $a_0 \succ a_1 \succ \dots \succ a_{n-1} \succ a$ in $L$ for some $a_0 \in L_0$ and for some $a_1, a_2, \dots, a_{n-1} \in L$ is $n = k$. Since $S$ is a maximal antichain in $L$, it follows that for all $b \in L \backslash S$, there is an element $a \in S$ such that $b \geq a$ or $b \leq a$. So, because $L$ is locally finite, it follows that for all $b \in L \backslash S$, there is an integer $k \in \mathbb{N}$ such that $b \in L_k$ or $b \in L_{-k}$. \\

For all $k \in \mathbb{Z}$, define $L^*_k = \bigcup_{i \leq k} L_i$. Firstly,
\begin{equation}\label{important}
L^*_k = \bigcup_{a \in L_k} \downarrow_L a. 
\end{equation}
For all $a \in L$, $\uparrow_L a$ is a sublattice of $L$ with exactly one minimal element. Hence, as $L$ is a locally finite sublattice of a free lattice, $\uparrow_L a$ is a locally finite sublattice of a free lattice with exactly one minimal element. So by Lemma \ref{first step}, $\uparrow_L a$ is countable. Hence, by symmetry, $\uparrow_L a$ and $\downarrow_L a$ are countable for all $a \in L$. \\

So suppose that $L$ is uncountable. Then, as $L = \bigcup_{k \in \mathbb{Z}} L^*_k$ and $\mathbb{Z}$ is countable, it follows that there exists an integer $k \in \mathbb{Z}$ such that $L^*_k$ is uncountable. Since $\downarrow_L a$ is countable for all $a \in L_k$, Equation \ref{important} implies, as the countable union of countable sets is countable, that $L_k$ is uncountable. Let $a \in L_k$. Because $L$ is a lattice, $a \wedge b$ exists in $L$ for all $b \in L_k$. Hence, consider the set
$$S_a = \{a \wedge b : b \in L_k \}. $$
Since $S_a \; \subseteq \;\, \downarrow_L a$, $S_a$ is countable because $\downarrow_L a$ is countable. Moreover, as $a \wedge b \leq b$ for all $b \in L_k$,
$$L_k \subseteq \bigcup_{c \in S_a} \uparrow_L c. $$
Since $L_k$ is uncountable and since $S_a$ is countable, the fact the any countable union of countable sets is countable implies that there exists an element $c \in S_a$ such that $\uparrow_L c$ is uncountable. But that is impossible by what we showed above. Hence, $L$ is countable.
\end{proof}

\section*{Acknowledgements}\label{sec:acknow} The author would like to thank Claude Laflamme and Robert Woodrow for their support during the writing of an earlier draft of this note, and the author would like to thank Stephanie van Willigenburg for her support during the writing of the current note.

\bibliographystyle{amsplain}

\end{document}